\documentclass[12pt]{amsart}
\usepackage{amsmath}
\usepackage{amsfonts}
\usepackage{amsthm}
\usepackage{amssymb}
\usepackage{amscd}
\usepackage[all]{xy}
\usepackage{enumerate}
\usepackage{hyperref}
\usepackage{mathrsfs}



\newcommand*{\hhom}{\mathcal{H}\kern -.7pt om}
\newcommand*{\ext}{\mathcal{E}\kern -.7pt xt}
\theoremstyle{plain}
\newtheorem{thm}{Theorem}[section]

\newtheorem{prop}[thm]{Proposition}

\newtheorem{cor}[thm]{Corollary}

\newtheorem{lem}[thm]{Lemma}

\theoremstyle{definition}
\newtheorem{defn}[thm]{Definition}

\newtheorem{expl}[thm]{Example}

\newtheorem{nota}[thm]{Notation}

\newtheorem*{ackn}{Acknowledgment}

\newtheorem{rmk}[thm]{Remark}
\newcommand{\sA}{\mathcal{A}}

\newcommand{\sC}{\mathcal{C}}

\newcommand{\sF}{\mathcal{F}}

\newcommand{\sL}{\mathcal{L}}

\newcommand{\sO}{\mathcal{O}}

\newcommand{\sS}{\mathcal{S}}
\newcommand{\sT}{\mathcal{T}}

\newcommand{\sX}{\mathcal{X}}
\newcommand{\sY}{\mathcal{Y}}

\newcommand{\mC}{\mathbb{C}}

\newcommand{\bu}{\mathbf{u}}
\newcommand{\tu}{\tilde{u}}
\newcommand{\btu}{\mathbf{\tilde{u}}}
\newcommand{\Ima}{\mathrm{Im}\,}
\newcommand{\vol}{\mathrm{Vol}\,}

\newcommand{\Spec}{\mathrm{Spec}\,}
\newcommand{\Ext}{\mathrm{Ext}}

\usepackage{color}

\numberwithin{equation}{section}

\newcommand{\beba}  {\begin{equation}\begin{array}{rcl}}

\newcommand{\eaee}  {\end{array}\end{equation}}

\makeatletter
\def\l@section{\@tocline{1}{0pt}{1pc}{}{}}
\def\l@subsection{\@tocline{2}{0pt}{1pc}{4.6em}{}}
\def\l@subsubsection{\@tocline{3}{0pt}{1pc}{7.6em}{}}
\renewcommand{\tocsection}[3]{%
  \indentlabel{\@ifnotempty{#2}{\makebox[2.3em][l]{%
    \ignorespaces#1 #2.\hfill}}}#3}
\renewcommand{\tocsubsection}[3]{%
  \indentlabel{\@ifnotempty{#2}{\hspace*{2.3em}\makebox[2.3em][l]{%
    \ignorespaces#1 #2.\hfill}}}#3}
\renewcommand{\tocsubsubsection}[3]{%
  \indentlabel{\@ifnotempty{#2}{\hspace*{4.6em}\makebox[3em][l]{%
    \ignorespaces#1 #2.\hfill}}}#3}
\makeatother

\title{Deformations of maps with fixed target and curvature of direct image bundles}

\author{Luca Rizzi}
\address{Luca Rizzi\\Center for Geometry and Physics, Institute for Basic Science (IBS)\\ Pohang, 37673, Korea
	\texttt{lucarizzi@ibs.re.kr}}

\begin{document}

\markboth{}{}

\begin{abstract}
	
	Consider a smooth proper holomorphic fibration of complex manifolds. It is known that the semipositivity of the curvature of the direct image of the relative canonical bundle can be read in terms of the cup product with the Kodaira-Spencer class of the fibers.
	
	Motivated by this result, in this paper we study the relation between deformations of maps with fixed target and the curvature of certain direct image bundles. As a result, we find some curvature formulas and use them to prove a seminegativity result for a vector bundle of relative forms naturally related to the deformation data.
\end{abstract}
\maketitle
\section{Introduction}
Let $f\colon \sX\to B$ be a smooth proper holomorphic fibration of complex manifolds of relative dimension $n$. It is classically known that the vector bundle $E$ associated to the direct image of the relative dualizing sheaf $f_*\omega_{\sX/B}$ is equipped with a natural hermitian metric with semipositive curvature in the sense of Griffiths, see  \cite{G1,gt} and also \cite{Fu}. More in details, consider the case where $B$ is a curve and denote by $X_t:=f^{-1}(t)$ the fiber over a point $t\in B$. The Chern curvature $\Theta$ for this hermitian metric satisfies the formula
\begin{equation}
	\label{griff}
	\langle\Theta u_t,u_t \rangle_t=\lVert\xi_t\cup u_t\lVert^2
\end{equation}
where $u_t\in E_t=H^0(X_t,\omega_{X_t})$ is an element in the fiber of $E$, $\xi_t\in H^1(X_t,T_{X_t})\simeq\Ext^1(\Omega^1_{X_t},\sO_{X_t})$ is the Kodaira-Spencer class of the deformation of $X_t$ induced by $f$, and $\cup$ is the cup product. See \cite{gt} or, for a more recent exposition, \cite{bern2}; see also Section \ref{sec4} of this paper.

So it is clear that when the Kodaira-Spencer class vanishes the curvature also vanishes, and, more in general, the holomorphic top forms $u_t$ with vanishing cup product with $\xi_t$ are exactly the directions where the curvature is not strictly positive.
Furthermore, counterexamples to the famous infinitesimal Torelli problem, see for example \cite{C}, show that  it is also possible that $\xi_t\neq 0$ but $\langle\Theta u_t,u_t \rangle_t=\lVert\xi_t\cup u_t\lVert^2=0$ for every $u_t\in E_t=H^0(X_t,\omega_{X_t})$, that is the infinitesimal (first order) deformation of $X_t$ is non trivial but the curvature is identically zero at $t$.

More recently, in \cite{bern1,bern2} the author studies the curvature of the vector bundle associated to the direct image $f_*(\omega_{\sX/B}\otimes \sL)$, where $\sL$ is a line bundle on $\sX$, and gives formulas analogous to (\ref{griff}) in this more general case. 
The computations rely on the notion of \emph{representative} of a section $u$ of  $f_*(\omega_{\sX/B}\otimes \sL)$.
A representative of $u$, here denoted by $\bu$, is a smooth $(n,0)$-form on $\sX$ with values in $\sL$ which restricts to $u_t$ on the fibers $X_t$. 
The representative of a section $u$ is not unique and one of the key points is to carefully choose a suitable representative in order to obtain more informative curvature formulas. 
In particular, if $\sL$ is semipositive, \cite{bern1} shows that $f_*(\omega_{X/B}\otimes \sL)$ is semipositive in the sense of Nakano, providing an evidence for the Griffiths conjecture, see \cite{G1}, while \cite{bern2} focuses on the relation between the curvature of $f_*(\omega_{X/B}\otimes \sL)$ and the Kodaira-Spencer class of the fibers $X_t$.

More broadly, there is a vast literature on the study of the curvature properties of (higher) direct images of sheaves of twisted relative forms and their subbundles; see for example  \cite{BPW,MT1,MT2,PT,T} and the therein quoted bibliography, just to name a few.

Motivated by these results relating deformations of complex manifolds to curvature of the direct image of the relative canonical sheaf, in this paper we study the relation between the so-called deformations of maps with fixed target and the curvature of certain direct image bundles. As the main results of this study, we find some curvature formulas and use them  to prove a seminegativity result for a vector bundle of relative forms naturally related to the deformation data.

We work locally, so denote by $\sX$  a smooth $n+1$-dimensional K\"ahler manifold, $S$ a smooth compact projective curve and $\Delta\subset \mC$ a complex disk.

We consider a proper holomorphic fibration with connected fibers $G\colon \sX\to \Delta\times S$. In this paper we study the case where $G$ is semistable; see Section \ref{sez0} or \cite{Il} for the definition. Nevertheless many results can be extended to a more general setting.   Call $f:=p_{\Delta}\circ G\colon \sX\to \Delta$ the composition of $G$ with the projection over  $\Delta$ and $\varphi:=p_{S}\circ G\colon \sX\to S$ the composition of $G$ with the projection over  $S$. We denote as before by $X_t$ the $n$-dimensional fiber $f^{-1}(t)$, $t\in\Delta$, and assume that these fibers are smooth. 


The morphism $G$ induces deformations of maps with fixed target $S$:  the  map $g_t:=\varphi_{|X_t}\colon X_t\to S$ at the point $t\in \Delta$ is deformed to the map $g_{t'}:=\varphi_{|X_{t'}}\colon X_{t'}\to S$. In Section \ref{sec2} we recall the fundamentals on deformations of maps with fixed target, for all the details we refer to \cite{sernbook,sern}. Here recall only that this deformation gives, for each $t\in \Delta$, a class $\zeta_t\in \Ext^1(\Omega^1_{X_t/S},\sO_{X_t})$ which can be seen as the analogue of the Kodaira-Spencer class. Furthermore, by fixing a point $s\in S$, the restriction $f_{|\varphi^{-1}(s)}\colon \varphi^{-1}(s)\to \Delta$ gives a deformation (in the usual sense) of the fibers $X_{t,s}:=G^{-1}(t,s)$, with the associated Kodaira-Spencer classes denoted by $\phi_s(\zeta_t)\in \Ext^1(\Omega^1_{X_{t,s}},\sO_{X_{t,s}})$. 

Now consider the sheaf of holomorphic relative differential 1-forms for the morphism $G$ denoted by $\Omega^1_{\sX/\Delta\times S}$ and defined by the exact sequence 
\begin{equation*}
	0\to  f^*\omega_\Delta\oplus \varphi^*\omega_S\to \Omega^1_\sX\to \Omega^1_{\sX/\Delta\times S}\to 0.
\end{equation*}
It turns out that the vector bundle $E$ on $\Delta$ relevant for the study of the deformation of maps given by $G$ is the vector bundle associated to the direct image $f_*(\Omega^{n-1}_{\sX/\Delta\times S}/{\text{Tors}})$, where  $\Omega^{n-1}_{\sX/\Delta\times S}:=\bigwedge^{n-1}\Omega^1_{\sX/\Delta\times S}$ and ${\text{Tors}}$ is its torsion part.

The bundle  $E$ can also be seen as a subbundle of the vector bundle $F$ associated to $f_*(\omega_{\sX/\Delta}\otimes \varphi^*\omega_S^\vee)$ and this provides a connection with the aforementioned works \cite{bern1,bern2} by taking  $\sL=\varphi^*\omega_S^\vee$. If $S$ is of 
general type then $\sL$ is seminegative, nevertheless many of the ideas still apply. For all the details, see Section \ref{sez3}.

In Section \ref{sec4} we define a natural hermitian metric on $E$ and compute the associated Chern connection and curvature. As in the case of \cite{bern1,bern2}, the choice of representative of a section $u$ of $E$ plays a key role. In particular, we consider different representatives depending on  different interpretations of $E$: if we see $E$ only as the vector bundle associated to the direct image $f_*(\Omega^{n-1}_{\sX/\Delta\times S}/{\text{Tors}})$ then a natural representative $\tu$ of $u$ is a smooth relative $(n-1,0)$-form for the fibration $\varphi$, on the other hand  if we see $E$ as a subbundle of $F$, then a suitable representative $\btu$ is a smooth $(n,0)$-form on $\sX$ with values in $\varphi^*\omega_S^\vee$. See Section \ref{subsez} for the precise definitions and the relation between $\tu$ and $\btu$.

Depending on the choice of these representatives we obtain various connection and curvature formulas for $E$. Notably, using the representative $\tu$, these formulas are related to the relative differentials of the fibration $\varphi\colon\sX\to S$ (i.e. differentiation along the fibres of $\varphi$) modulo torsion, denoted by $\bar{\partial}^0_{\sX/S},\partial^0_{\sX/S}$.
For example the following formulas hold for the curvature of $E$:
\begin{thm}\label{thm1}
	Take $u$ a section of $E$ and write $\partial\btu=\mu\wedge dt$ and $\bar{\partial}^0_{\sX/S}\tu=\tilde{\nu}\wedge d\bar{t}+\tilde{\eta}\wedge dt$. Then the curvature of $E$ on $u$ is given at $t$ by 
	\begin{equation*}
		\langle\Theta_E u_t,u_t \rangle_t=-\lVert P'_\perp{\mu_t}\lVert^2_t-2c_{n-1}\int_{X_t}\tilde{\eta}_t\wedge\overline{\tilde{\eta}_t}\wedge g^*_t\textnormal{Vol}S
	\end{equation*}  where $P'_\perp$ is the orthogonal projection on the orthogonal complement of $E_t$ and $c_{n-1}$ is a unimodular constant.
	
	If $\varphi$ is semistable with separated variables, we can also write  $\partial^0_{\sX/S}\tu=\tilde{\mu}\wedge dt$ and the curvature becomes
	\begin{equation*}
		\langle\Theta_E u_t,u_t \rangle_t=-\lVert P'_\perp\tilde{\mu}_t\lVert^2_t-2c_{n-1}\int_{X_t}\tilde{\eta}_t\wedge\overline{\tilde{\eta}_t}\wedge g^*_t\textnormal{Vol}S.
	\end{equation*}
\end{thm} See Theorem \ref{curv2} for all the details.

 Recall that in our case $F$ is the vector bundle associated to $f_*(\omega_{\sX/\Delta}\otimes \varphi^*\omega_S^\vee)$ so when $S$ is a curve of general type then $\sL=\varphi^*\omega_S^\vee$ is seminegative, so in general we cannot expect the curvature of $F$ and of  $E$ to be semipositive. In the above formula, while the first summand is always seminegative, the second could give a positive contribution. 

In Section \ref{sez5}, we finally relate the above curvature formulas to the theory of deformation of maps with fixed target. 

We assume for simplicity that $n=2$, that is $\sX$ has dimension 3 and each $g_t\colon X_t\to S$ is a fibred surface over a fixed curve.
Most of the results generalise to the higher dimensional general case, however the discussion becomes technically more complicated without giving substantial improvement.

 Note that in this case $E$ is the vector bundle associated to $f_*(\Omega^{1}_{\sX/\Delta\times S})$ and its fiber is $ E_t=H^0(X_t,\Omega^{1}_{X_t/S})$. Hence $u_t$ can be seen as a holomorphic relative 1-form for the morphism $g_t$; in particular $u_t$ associates to a general $s\in S$ a holomorphic 1-form on the fiber $X_{t,s}:=g_t^{-1}(s)$, we denote this form by  $u_{t,s}\in H^0(X_{t,s}, \Omega^1_{X_{t,s}})$. 
Similarly, the form $\tilde{\eta}_t$ introduced by Theorem \ref{thm1}, is a relative $(0,1)$-form for the morphism $g_t$, so it associates to a general $s\in S$ a $(0,1)$-form on the fiber $X_{t,s}$ denoted by $\tilde{\eta}_{t,s}$.

We prove that $\tilde{\eta}_t$ actually defines a global  holomorphic section of the vector bundle $R^1{g_t}_*\sO_{X_t}$ relating to the deformation data as follows:  the  $(0,1)$-form  $\tilde{\eta}_{t,s}$ defines a cohomology class in $H^1(X_{t,s},\sO_{X_{t,s}})$ and it holds that $[\tilde{\eta}_{t,s}]=\phi_s(\zeta_t)\cup u_{t,s}$, where we recall that $\zeta_t\in \Ext^1(\Omega^1_{X_t/S},\sO_{X_t})$ is the   class associated to the deformation of maps with fixed target.
%

Theorem \ref{thm1} becomes:
 \begin{prop}\label{prop2}
 If $n=2$, the curvature of $E$ on $u$ is given at $t$ by 
 	\begin{equation*}
 		\langle\Theta_E u_t,u_t \rangle_t=-\lVert P'_\perp{\mu_t}\lVert^2_t{+2c_{n-1}\int_{X_t}|\tilde{\eta}_{t}|^2 g^*_t\textnormal{Vol}S}
 	\end{equation*}  where $\tilde{\eta}_t$ is such that $[\tilde{\eta}_{t,s}]=\phi_s(\zeta_t)\cup u_{t,s}$.
 \end{prop}
 See Proposition \ref{curvdef}. 
 
From this formula we see that the sections $u$ of $E$ such that $u_{t,s}$ has vanishing cup product with the class $\phi_s(\zeta_t)$ for general $s\in S$ give directions where the curvature of $E$ is seminegative.
These sections actually identify a vector bundle as follows. In Section \ref{sez5} we construct a morphism
\begin{equation*}
	f_*(\Omega^1_{\sX/\Delta\times S})\to  {p_{\Delta}}_*(R^1G_*\sO_\sX)\otimes \omega_{\Delta}
\end{equation*}
and, recalling that $E$ is the vector bundle associated to $f_*(\Omega^1_{\sX/\Delta\times S})$, the kernel of this morphism defines a bundle $K$ whose sections are exactly  the elements $u$ of $E$ such that $u_{t,s}$ has vanishing cup product with the class $\phi_s(\zeta_t)$ for general $s\in S$. Hence we have the following seminegativity result.
\begin{thm}\label{thm3}
	The vector bundle	$K$ is seminegatively curved.
\end{thm} See Theorem \ref{semineg}.

%
\begin{ackn}
	This work was supported by the Institute for Basic Science (IBS-R003-D1). 
	The author is a member of INdAM-GNSAGA.
\end{ackn}

\section{Setting and notation}\label{sez0}
In this section we recall some key definitions and fix the notation. 

Given a holomorphic fibration with connected fibers between smooth complex manifolds $h\colon Y\to Z$, we denote by $\omega_{Y/Z}:=\omega_Y\otimes h^*\omega_Z^\vee$ the relative canonical sheaf  of $h$, and by $\Omega^1_{Y/Z}$ the sheaf of relative holomorphic differential 1-forms defined by the short exact sequence 
\begin{equation}\label{seqrelgen}
0\to h^*\Omega^1_Z\to \Omega^1_Y\to \Omega^1_{Y/Z}\to 0.
\end{equation} We also denote by $\Omega^p_{Y/Z}:=\bigwedge^p \Omega^1_{Y/Z}$ the sheaf of relative holomorphic differential $p$-forms. 

We will mostly work with semistable fibrations, following the definition by Illusie \cite{Il}, see also \cite{KKMSD} for the case $\dim Z=1$.
Let $W$ be a normal crossing divisor on $Y$ and $D$ a normal crossing divisor on $Z$ such that $W=h^{-1}(D)$. Take local coordinates $y_1,\dots,y_n$ on $Y$ such that $W$ is locally given by $y_1\ldots y_k=0$ and local coordinates $z_1,\dots,z_m$ on $Z$ such that $D$ is $z_1\ldots z_p=0$.
Under these assumptions, a semistable map $h$ is given locally by 
\begin{equation*}
	h(y_1,\dots,y_n)=(y_1\cdots y_{s_1},y_{s_1+1}\cdots y_{s_2},\dots,y_{s_{p-1}+1}\cdots y_{s_p},y_{k+1},y_{k+2},\dots,y_{k+{m-p}})
	\label{semistable}
\end{equation*}  with $s_p=k$.
In particular the fibers of $h$ are reduced.

In this situation it is often convenient to consider the logarithmic version of (\ref{seqrelgen}), which is an exact sequence of vector bundles on $Y$
\begin{equation*}\label{seqloggen}
0\to h^*\Omega^1_{Z}(\log D)\to \Omega^1_Y(\log W)\to \Omega^1_{Y/Z}(\log)\to 0.
\end{equation*}
 We briefly recall that if $W$ is locally given by $y_1y_2\cdots y_k=0$ in a suitable local coordinate system, the sheaf $\Omega^1_Y(\log W)$ of {logarithmic differentials} is the locally free $\sO_Y$-module generated by $dy_1/y_1,\ldots,dy_k/y_k,dy_{k+1},\ldots,dy_{n}$; see \cite{De}. 

\begin{rmk}
If $h$ is semistable, $\Omega^1_{Y/Z}$ is torsion free but in general not locally free.
The sheaf $\Omega^1_{Y/Z}$ is in fact a subsheaf, but in general not a subbundle, of the vector bundle $\Omega^1_{Y/Z}(\log)$, see also \cite{R2}.
On the other hand, $\Omega^p_{Y/Z}$ may have torsion elements for $p>1$; in this case there is a morphism $\Omega^p_{Y/Z}\to \Omega^p_{Y/Z}(\log)$ which is not necessarily injective. 
 It is not difficult to see that $\det \Omega^1_{Y/Z}(\log)=\omega_{Y/Z}$.
 
Of course if $h$ is smooth, $\Omega^p_{Y/Z}=\Omega^p_{Y/Z}(\log)$ is a vector bundle for all $p$ and $\det \Omega^1_{Y/Z}=\omega_{Y/Z}$.
\end{rmk}

\section{Preliminaries on deformation of maps with fixed target}\label{sec2}
In this section we recall the fundamentals of the theory of deformations of maps with fixed target. For all the details we refer to \cite{sernbook,sern}, see also \cite{H1,H2,H3}; here we recall the key properties that we will need in the following. 

Let $Y$ be a compact complex manifold. It is well known that the first order deformations of $Y$, that is flat surjective morphisms $\sY\to \Spec(\mC[t]/(t^2))$ such that $Y$ is the closed fiber over the closed point $\Spec\mC\hookrightarrow\Spec(\mC[t]/(t^2))$, are classified by the cohomology group $H^1(Y,T_Y)$. This group is isomorphic to the extension group $\Ext^1(\Omega^1_Y,\sO_Y)$; more explicitly, to a deformation $\sY\to \Spec(\mC[t]/(t^2))$ we associate the following extension of $\Omega^1_Y$ by $\sO_Y$:
$$
0\to\sO_Y\to \Omega^1_{\sY|Y}\to \Omega^1_Y\to 0.
$$

Given a morphism $h\colon Y\to Z$ between two complex varieties, a (first order) deformation of $h$ leaving the target fixed is given by a cartesian diagram
\begin{equation}\label{ext}
\xymatrix{
Y\ar[r]\ar^h[d]&\sY\ar^H[d]\\
Z\ar[r]\ar[d]&Z\times \Spec(\mC[t]/(t^2))\ar^p[d]\\
\Spec\mC\ar[r]&\Spec(\mC[t]/(t^2))
}
\end{equation} where $p$ is the projection on the second factor and $p\circ H$ is flat.

From now on we consider the case when $h$ is a semistable fibration, since this is the relevant case for this paper.
In this case, a first order deformation of  $h$ leaving the target fixed is associated to the element of the group $\Ext^1(\Omega^1_{Y/Z},\sO_Y)$ given by the exact sequence
\begin{equation}
	\label{ext1}
	0\to \sO_Y\to\Omega^1_{\sY/Z|Y}\to \Omega^1_{Y/Z}\to 0.
\end{equation}

Ignoring the middle row of Diagram (\ref{ext}), a deformation of $h$ gives a deformation of $Y$ in the classical sense recalled above. Formally, this is expressed by the homomorphism
\begin{equation}\label{defmap}
	\Ext^1(\Omega^1_{Y/Z},\sO_Y)\to \Ext^1(\Omega^1_{Y},\sO_Y)
\end{equation}
 obtained by applying the Hom functor to the exact sequence (\ref{seqrelgen}).
So essentially, a deformation of $h$ with fixed target consists of a morphism $H$ between a deformation of $Y$ and the trivial deformation of $Z$ and an identification of $h$ with the restriction of $H$ to the closed fiber.

 Note that a deformation of $h$ leaving the target fixed also induces a deformation $\sS_z$ of a fiber $Y_z:=h^{-1}(z)$, $z\in Z$, by the following diagram
\begin{equation}\label{deffib}
	\xymatrix{
	Y_z\ar[r]\ar[d]&	\sS_z\ar[r]\ar[d]&\sY\ar^H[d]\\
		\Spec\mC\ar[r]&\Spec(\mC[t]/(t^2))\ar^-{\{z\}\times id}[r]&Z\times \Spec(\mC[t]/(t^2))
	}
\end{equation}
In particular we have a homomorphism for all $z\in Z$
\begin{equation}
	\label{phi}
\phi_z\colon	\Ext^1(\Omega^1_{Y/Z},\sO_Y)\to \Ext^1(\Omega^1_{Y_z},\sO_{Y_z}).
\end{equation}

The group $\Ext^1(\Omega^1_{Y/Z},\sO_Y)$ has another useful interpretation. Recall the definition of the relative extension sheaf $\ext^1_h$.
\begin{defn}
	Given a morphism $h\colon Y\to Z$, the relative extension sheaf $\ext^p_h$ is  the $p$-th derived functor of $h_*\hhom$.
\end{defn} 
For all the properties of the relative extension sheaves we refer to \cite[Chapter 1]{Bir}. 

In our case, we consider the sheaf (on $Z$) $\ext^1_h(\Omega^1_{Y/Z},\sO_Y)$ since its cohomology controls the deformation theory of $h$.
We recall the following results without proof, see \cite[Lemma 4.4 and Lemma 4.6]{RZ}.
\begin{lem}
	\label{generalg}
	There is an injective morphism of sheaves
	\begin{equation*}
		R^1h_*T_{Y/Z}\hookrightarrow \ext^1_h(\Omega^1_{Y/Z},\sO_Y)
	\end{equation*}
	which is an isomorphism over an open dense subset of $Z$. In particular, for general $z\in Z$ we have the isomorphism
	\begin{equation*}
		\ext^1_h(\Omega^1_{Y/Z},\sO_Y)\otimes \mC(z)\cong H^1(Y_z,T_{Y_z})\cong \textnormal{Ext}^1(\Omega^1_{Y_z},\sO_{Y_z}).
	\end{equation*}
	The morphism $T_Z\to\ext^1_h(\Omega^1_{Y/Z},\sO_Y)$ obtained by applying the functor $h_*\hhom(-,\sO_Y)$ to (\ref{seqrelgen}) induces the Kodaira-Spencer map 
	$$
	T_{Z,z}\to H^1(Y_z,T_{Y_z})
	$$ at the general point of $Z$.
\end{lem}
For this reason the morphism $T_Z\to\ext^1_h(\Omega^1_{Y/Z},\sO_Y)$  is called Global Kodaira-Spencer map of the family $h$.

The space of global sections of $\ext^1_h(\Omega^1_{Y/Z},\sO_Y)$ encodes the deformation theory of $h$ by the following result, see \cite{sern}.

\begin{lem}
	\label{ss2}
	Let $h\colon Y\to Z$ be a fibration, then we have a  homomorphism
	\begin{equation*}
	 \textnormal{Ext}^{1}(\Omega^1_{Y/Z},\sO_Y)\to H^0(Z,\ext^1_h(\Omega^1_{Y/Z},\sO_Y)).
	\end{equation*}
	 If $Z$ is a curve, this homomorphism is surjective and if $h_*T_{Y/Z}=0$,  it is an isomorphism (even if $Z$ is not a curve). 
	 
	 Its composition with the evaluation at the general fiber
	 $$
	 \textnormal{Ext}^{1}(\Omega^1_{Y/Z},\sO_Y)\to H^0(Z,\ext^1_h(\Omega^1_{Y/Z},\sO_Y))\to \textnormal{Ext}^1(\Omega^1_{Y_z},\sO_{Y_z})
	 $$ is the morphism $\phi_z$ from (\ref{phi}).
\end{lem}

\begin{rmk}
	The condition $h_*T_{Y/Z}=0$ is for example satisfied 
	when the fibers of $h$ are of general type. In this case we can identify $\textnormal{Ext}^{1}(\Omega^1_{Y/Z},\sO_Y)\cong H^0(Z,\ext^1_h(\Omega^1_{Y/Z},\sO_Y))$. 
	Hence a deformation of $h$ with fixed target is uniquely associated to a global section over $Z$ of the sheaf $\ext^1_h(\Omega^1_{Y/Z},\sO_Y)$. This section associates to a general $z\in Z$ the Kodaira-Spencer class of the deformation of $Y_z$  defined in (\ref{deffib}).
\end{rmk}

For the purpose of this paper, we consider geometric deformations, that is deformations defined over a  disk $\Delta\subset \mC$. So for example Diagram (\ref{ext}) becomes
\begin{equation*}
	\xymatrix{
		Y\ar[r]\ar^h[d]&\sY\ar^H[d]\\
		Z\ar[r]\ar[d]&Z\times \Delta \ar^p[d]\\
		\Spec\mC\ar[r]&\Delta
	}
\end{equation*} Of course, exactly as in the case of deformation of varieties, a deformation over a disk defines a first order deformation but the viceversa is not true in general due to the existence of obstructions. For more details on the obstruction theory in the case of deformation of maps with fixed target see \cite{sernbook}.

\section{Vector bundles of relative forms associated to a deformation of maps with fixed target}
\label{sez3}
The goal of this section is to consider a geometric deformation of maps with fixed target and define suitable vector bundles of relative differential forms whose curvature formulas, as we will see in later sections, are related to the deformation data.

\subsection{Deformations of semistable fibrations}
Denote by $\sX$  a smooth $n+1$-dimensional K\"ahler manifold, $S$ a smooth compact projective curve and $\Delta\subset \mC$ a small complex disk.

We consider a proper holomorphic fibration with connected fibers $G\colon \sX\to \Delta\times S$. We call $f$ its composition with the projection over  $\Delta$
$$
f\colon \sX\to \Delta\times S\to \Delta
$$ and denote by $X_t$ the $n$-dimensional fiber $f^{-1}(t)$, $t\in\Delta$. Up to restriction of $\Delta$ we  assume that these fibers are smooth. The composition of $G$ with the projection over $S$ will be denoted by $\varphi$:
$$
\varphi\colon \sX\to \Delta\times S\to S.
$$

In this setting, we have a the deformation of maps with fixed target $S$: for $t\in \Delta$, the holomorphic map $g_t:=\varphi_{|X_t}\colon X_t\to S$ can be interpreted as a deformation of the map $g_0\colon X_0\to S$ over the base point $0\in \Delta$.

In this paper we study the case where $G$ is semistable, nevertheless many results can be extended to a more general setting. 
The following Lemma identifies a convenient choice of local coordinates that we often implicitly use for local computations around the critical points of $G$.
\begin{lem}In our setting, choose $t,s$ local coordinates in $\Delta$ and $S$ respectively. 
There exist a change of coordinates $t'=t$ and $s'=s-l(t)$ on $\Delta\times S$ and local coordinates $x_1,\dots,x_{n+1}$ on $\sX$ such that $G\colon \sX\to \Delta\times S$ is locally given by 
$$
t'=x_{n+1},\quad s'=x_1\dots x_p
$$ with $p\leq n$.
\end{lem}
\begin{proof}
Since $G$ is semistable, take the divisor of singular values $D\subset \Delta\times S$. It is not difficult to see that $D$ cannot be singular in $(t,s)=(0,0)$, otherwise $f\colon \sX\to \Delta$ is not smooth. So if locally $D$ is given by an equation $d=0$, at least one of the derivatives $\partial d/\partial t$ and $\partial d/\partial s$ does not vanish in $(0,0)$.

If $\partial d/\partial t\neq 0$ and $\partial d/\partial s=0$ at the origin, once again $f$ is not smooth.

So we are left with the case $\partial d/\partial s\neq 0$. Here by the implicit function  theorem we can express $D$ as $s=l(t)$ for some holomorphic function $l$. We then choose coordinates $t'=t$ and $s'=s-l(t)$ so that $D$ is  $s'=0$ in the new coordinates. The claim easily follows from the definition of semistable fibration.
\end{proof}

In particular, $g_t\colon X_t\to S$ is semistable for every $t$  hence the $g_t$-fibers over a point in $S$ are either smooth or reduced and normal crossing divisors in $X_t$. On the other hand note that $\varphi\colon \sX\to S$ is not necessarily semistable, as the following example shows.

\begin{expl}\label{example1}
We point out that we basically have three cases for $\varphi$ depending on $l(t)$ and its derivative. In fact locally $\varphi$ is given by $\varphi(x_1,\dots,x_{n+1})=(x_1\dots x_p+l(x_{n+1}))$ and
\begin{enumerate}
\item if $\partial l/\partial t\neq 0$ at the origin, then $\varphi$ is smooth in this neighbourhood.
\item if $\partial l/\partial t= 0$ at the origin, then $\varphi$ is not necessarily semistable.
\item  if however $l(t)\equiv 0$ then $\varphi$ is locally given by $\varphi(x_1,\dots,x_{n+1})=(x_1\dots x_p)$ and hence it is semistable.
\end{enumerate} 
\end{expl}

\subsection{The vector bundles of relative forms}
Recall the definitions of the relative canonical sheaves $\omega_{\sX/\Delta}:=\omega_{\sX}\otimes f^*\omega_{\Delta}^\vee$, $\omega_{\sX/S}:=\omega_{\sX}\otimes \varphi^*\omega_{S}^\vee$ and $\omega_{\sX/\Delta\times S}:=\omega_{\sX}\otimes (f^*\omega_{\Delta}\otimes \varphi^*\omega_S)^\vee$. These are line bundles on $\sX$.

On the other hand the sheaves of relative differential forms $\Omega^1_{\sX/\Delta}$, $\Omega^1_{\sX/S}$ and $\Omega^1_{\sX/\Delta\times S}$ are defined respectively by the short exact sequences on $\sX$
\begin{equation}
\label{xsudelta}
0\to f^*\omega_\Delta\to \Omega^1_\sX\to \Omega^1_{\sX/\Delta}\to 0
\end{equation}
\begin{equation}
\label{xsus}
0\to \varphi^*\omega_S\to \Omega^1_\sX\to \Omega^1_{\sX/S}\to 0
\end{equation} and 
\begin{equation}
\label{xsudeltaes}
0\to f^*\omega_\Delta\oplus \varphi^*\omega_S\to \Omega^1_\sX\to \Omega^1_{\sX/\Delta\times S}\to 0.
\end{equation}

From the inclusion $\Omega^1_{\sX/\Delta\times S}\hookrightarrow \Omega^1_{\sX/\Delta\times S}(\log)$ recalled in Section \ref{sez0}, we have a morphism of the top wedge product 
\begin{equation}
	\label{contors}
	\Omega^{n-1}_{\sX/\Delta\times S}\to\Omega^{n-1}_{\sX/\Delta\times S}(\log)= \omega_{\sX/\Delta\times S}=\omega_{\sX/\Delta}\otimes \varphi^*\omega_S^\vee
\end{equation} which is an isomorphism on an open subset of $\sX$ and becomes injective after removing the torsion
\begin{equation}\label{senzators}
\Omega^{n-1}_{\sX/\Delta\times S}/{\text{Tors}}\hookrightarrow \omega_{\sX/\Delta\times S} =\omega_{\sX/\Delta}\otimes \varphi^*\omega_S^\vee.
\end{equation}


Since $f$ is smooth, the $n$-th wedge product of the exact sequence (\ref{xsudelta}) tensored by the line bundle $\varphi^*\omega_S^\vee$ gives the following exact sequence
\begin{equation}
\label{approcciobo}
0\to \Omega^{n-1}_{\sX/\Delta}\otimes f^*\omega_\Delta\otimes \varphi^*\omega_S^\vee\to \Omega^n_\sX\otimes \varphi^*\omega_S^\vee\to \omega_{\sX/\Delta}\otimes \varphi^*\omega_S^\vee\to 0.
\end{equation}
At the same time, the exact sequences (\ref{xsus}) and (\ref{xsudeltaes}) fit in the diagram
\begin{equation*}
\xymatrix{
0\ar[r]& \varphi^*\omega_S\ar[r]\ar@{^{(}->}[d]& \Omega^1_\sX\ar[r]\ar@{=}[d]& \Omega^1_{\sX/S}\ar[r]\ar@{>>}[d]& 0\\
0\ar[r]&f^*\omega_\Delta\oplus \varphi^*\omega_S\ar[r]& \Omega^1_\sX\ar[r]& \Omega^1_{\sX/\Delta\times S}\ar[r]& 0
}
\end{equation*} which shows  that  on $\sX$ we have the following exact sequence
\begin{equation}\label{1forme}
0\to f^*\omega_\Delta\to \Omega^1_{\sX/S}\to \Omega^1_{\sX/\Delta\times S}\to 0.
\end{equation} whose $(n-1)$-wedge product is the sequence
\begin{equation}
\label{approccionuovo}
f^*\omega_\Delta\otimes \Omega^{n-2}_{\sX/\Delta\times S}\to \Omega^{n-1}_{\sX/S}\to \Omega^{n-1}_{\sX/\Delta\times S}\to0
\end{equation} where in general we lose left exactness. Note that the restriction of (\ref{1forme}) to $X_t$ gives the extension associated to the deformation of maps with fixed target (\ref{ext1}) at the point $t$.

Sequences (\ref{approcciobo}) and (\ref{approccionuovo}) fit into the commutative diagram
\begin{equation}\label{diagrammaconfronto}
\xymatrix{
0\ar[r]& \Omega^{n-1}_{\sX/\Delta}\otimes f^*\omega_\Delta\otimes \varphi^*\omega_S^\vee\ar[r]&\Omega^n_\sX\otimes \varphi^*\omega_S^\vee\ar[r]& \omega_{\sX/\Delta}\otimes \varphi^*\omega_S^\vee\ar[r]& 0\\
& f^*\omega_\Delta\otimes \Omega^{n-2}_{\sX/\Delta\times S}\ar[r]\ar[u]& \Omega^{n-1}_{\sX/S}\ar[r]\ar[u]& \Omega^{n-1}_{\sX/\Delta\times S}\ar[r]\ar[u]& 0
}
\end{equation}
where the rightmost vertical arrow is (\ref{contors}).
The other vertical morphisms are given by wedge product, e.g. 
$\Omega^{n-1}_{\sX/S}\to \Omega^n_\sX\otimes \varphi^*\omega_S^\vee$ comes from the wedge product $\Omega^{n-1}_{\sX/S}\otimes \varphi^*\omega_S\to \Omega^n_\sX$. 

%

In \cite{bern1} and \cite{bern2} the author studies the curvature of the vector bundle associated to the direct image $f_*(\omega_{\sX/\Delta}\otimes \sL)$ where $\sL$ is a line bundle on $\sX$ equipped with a smooth metric of
semipositive curvature. In our setting the direct image $f_*\omega_{\sX/\Delta\times S}=f_*(\omega_{\sX/\Delta}\otimes \varphi^*\omega_S^\vee)$ has of course the same structure with $\sL=\varphi^*\omega_S^\vee$. If $S$ is of general type then $\sL$ is seminegative, but nevertheless many formulas coming from \cite{bern1} and \cite{bern2} still hold.

 On the other hand by (\ref{senzators}) we have an injective morphism $f_*(\Omega^{n-1}_{\sX/\Delta\times S}/{\text{Tors}})\hookrightarrow f_*(\omega_{\sX/\Delta}\otimes \varphi^*\omega_S^\vee)$. Actually, up to restriction  of $\Delta$, we assume that it gives an inclusion of vector bundles. 
\begin{defn}\label{notationE}
We denote by $F$ the vector bundle associated to  $f_*(\omega_{\sX/\Delta}\otimes \varphi^*\omega_S^\vee)$ and by $E$ the subbundle associated to $f_*(\Omega^{n-1}_{\sX/\Delta\times S}/{\text{Tors}})$.
\end{defn}
 We will give curvature formulas for $E$ and show how they are related  to the theory of deformations of maps with fixed target.

Note that the fiber $F_t$ over a point $t\in \Delta$ is $F_t=H^0(X_t,\omega_{X_t/S})=H^0(X_t,\omega_{X_t}\otimes g_t^*\omega_S^\vee)$. On the other hand the fiber $E_t$ is the image 
\begin{equation}\label{Efiber}
	E_t=\Ima(H^0(X_t,\Omega^{n-1}_{X_t/S}/{\text{Tors}})\to H^0(X_t,\omega_{X_t/S})).
\end{equation}
In the special case where $g_t$ is smooth then  $\Omega^{n-1}_{X_t/S}$ and $\omega_{X_t/S}$ coincide so that $E=F$, but this is not always the case when $g_t$ is semistable.


\section{Metric, connection and curvature on vector bundles of relative forms}\label{sec4}

In this section we define a suitable hermitian metric on $E$ and give formulas for the associated Chern connection and curvature.
\subsection{Different interpretations of sections of $E$ and of the hermitian metric}
We start by recalling the natural hermitian metric on  $F$ as defined in \cite{bern1,bern2}.
Consider $u$ a section of $F$. It can be seen as a function that maps a general point $t\in \Delta$ to a global holomorphic $(n,0)$-form on $X_t$ with values in $g_t^*\omega_S^\vee$. That is, denoting this $(n,0)$-form by $u_t$,  $u_t\in H^0(X_t,\omega_{X_t}\otimes g_t^*\omega_S^\vee)=H^0(X_t,\omega_{X_t/S})$.
The hermitian norm on $F$ is defined by
\begin{equation}
\label{norma1}
\lVert u_t \lVert^2_{F,t}:=\int_{X_t} c_n u_t\wedge\bar{u}_te^{-\phi}
\end{equation} where $c_n=i^{{n}^2}$ is a unimodular constant and $\phi$ is the weight of a metric on the line bundle $g_t^*\omega_S^\vee$. Of course in our case we consider the metric coming from the K\"ahler structure of $S$, which from now on we assume fixed.

For a section of $E$ the situation is more articulated as highlighted by the following remark.
\begin{rmk}\label{keyrmk}
When $u$ is a section of $E$, we have two parallel interpretations of $u_t$:
\begin{enumerate}
	\item since $E\subseteq F$, we can regard $u_t$ as above, i.e. a global holomorphic $(n,0)$-form on $X_t$ with values in $g_t^*\omega_S^\vee$
	\item since the fiber of $E$ is  $E_t=\Ima (H^0(X_t,\Omega^{n-1}_{X_t/S}/{\text{Tors}})\to H^0(X_t,\omega_{X_t/S}))$, $u_t$ can be seen as a global (non torsion) relative holomorphic $(n-1)$-form for the morphism $g_t$. 
\end{enumerate} 
\end{rmk}
We will be particularly concerned with the second interpretation since it will allow us to relate the metric properties of $E$ to the deformation of maps with fixed target. On the other hand the first relates our discussion to \cite{bern1,bern2}.
\begin{rmk}\label{remarkloc}
	This different interpretation amounts to a different local expression of $u_t\in E_t$.  For example, on a open subset of $X_t$ where $g_t$ is smooth, we take local coordinates $x_1,\dots,x_{n-1},s$ so that $g_t$ is given by projection on the last coordinate.
	If we use the first interpretation, locally $u_t=adx_1\wedge\dots\wedge dx_{n-1}\wedge ds\otimes \frac{\partial}{\partial s}$ for some holomorphic function $a$. On the other hand, if we use the second interpretation, we write $u_t=adx_1\wedge\dots\wedge dx_{n-1}$. The contraction of $ds\otimes \frac{\partial}{\partial s}$ gives exactly the isomorphism  between $\Omega^{n-1}_{X_t/S}$ and $\omega_{X_t/S}=\omega_{X_t}\otimes g_t^*\omega_S^\vee$ on the open set where $g_t$ is smooth.
\end{rmk}

If we consider the first interpretation of $u_t$ as above, it is natural to consider on $E$ the hermitian metric induced by restriction from $F$. On the other hand, if we consider the second interpretation
 of $u_t$ it is natural to define the following norm:
\begin{equation}
\label{norma2}
\lVert u_t \lVert^2_{E,t}:=\int_{X_t} c_{n-1} u_t\wedge\bar{u}_t\wedge g_t^*\vol S
\end{equation} where $\vol S$ is the volume form on $S$ associated to its K\"ahler metric.

We show that the metric (\ref{norma2}) coincides with the one induced by (\ref{norma1}) up to a constant.
\begin{prop}\label{normeuguali}
If $u_t\in E_t\subseteq F_t$ then $\lVert u_t \lVert^2_{F,t}=2\lVert u_t \lVert^2_{E,t}$.
\end{prop}
\begin{proof}
We show this with a local computation on an open subset of $X_t$ where the morphism $g_t$ is smooth, but the computation is similar also around the critical points.

As in Remark \ref{remarkloc}, we take local coordinates $x_1,\dots,x_{n-1},s$ on $X_t$ so that $g_t$ is given by projection on the last coordinate. We write the volume form on $S$ as $\frac{i}{2}\alpha ds\wedge d\bar{s}$, then $e^{-\phi}=\alpha$. 

Take an element $u_t\in E_t\subseteq F_t$. As we have seen above, we have two different interpretations of $u_t$: as holomorphic $(n,0)$-form with values in $g_t^*\omega_S^\vee$ and as  relative holomorphic $(n-1)$-form for the morphism $g_t$. 

In (\ref{norma1}) we are considering the first interpretation, so $u_t=adx_1\wedge\dots\wedge dx_{n-1}\wedge ds\otimes \frac{\partial}{\partial s}$ and
$$
c_nu_t\wedge\bar{u}_te^{-\phi}=c_n|a|^2 dx_1\wedge\dots\wedge dx_{n-1}\wedge ds\wedge d\bar{x}_1\wedge\dots\wedge  d\bar{x}_{n-1}\wedge d\bar{s}\cdot\alpha.
$$ In (\ref{norma2}) we are considering the second interpretation so $u_t=adx_1\wedge\dots\wedge dx_{n-1}$, and
\begin{multline*}
c_{n-1} u_t\wedge\bar{u}_t\wedge g_t^*\vol S=c_{n-1}|a|^2 dx_1\wedge\dots\wedge dx_{n-1}\wedge d\bar{x}_1\wedge\dots\wedge  d\bar{x}_{n-1}\wedge \frac{i}{2}\alpha ds\wedge d\bar{s}=\\
=c_{n-1}\frac{i}{2}(-1)^{n-1}|a|^2\alpha dx_1\wedge\dots\wedge dx_{n-1}\wedge ds\wedge d\bar{x}_1\wedge\dots\wedge  d\bar{x}_{n-1}\wedge d\bar{s}.
\end{multline*}
Since $c_n=c_{n-1}\cdot i\cdot(-1)^{n-1}$, the proof is complete.
\end{proof}
Thanks to this result, from now on we will simply write $\lVert u_t \lVert^2_{t}:=\lVert u_t \lVert^2_{F,t}=2\lVert u_t \lVert^2_{E,t}$ for the norm of a section of $E$.
\begin{rmk}
	If $v\in H^0(X_t,\Omega^{n-1}_{X_t/S})$ is a torsion element, then 
	$$
	\int_{X_t} c_{n-1} v\wedge\bar{v}\wedge g_t^*\vol S=0,
	$$ so $\lVert \cdot \lVert^2_{t}$ is not a norm on $H^0(X_t,\Omega^{n-1}_{X_t/S})$. Since by definition $E$ is the vector space associated to $f_*(\Omega^{n-1}_{\sX/\Delta\times S}/{\text{Tors}})$, $\lVert \cdot \lVert^2_{t}$ is a well defined norm on $E_t$ where the torsion is removed.
\end{rmk}
\subsection{Connection and curvature of $F$}
From \cite{bern1,bern2}, we have an explicit expression for the  Chern connection $D_F = D_F'+ D_F''$ and curvature $\Theta_F$ of the vector bundle $F$ endowed with the hermitian metric (\ref{norma1}). See also \cite{G1,gt} for the original untwisted case $f_*\omega_{\sX/\Delta}$. 

These explicit expressions of connection and curvature rely on the notion of \emph{representative} of a section of $F$. 
If $u$ is a (non necessarily holomorphic) section of $F$, a representative of $u$, here denoted by $\bu$, is a smooth $(n,0)$-form on $\sX$ with values in $\varphi^*\omega_S^\vee$ which restricts to $u_t$ on the fibers $X_t$. 
This means that $i^*_t (\bu) = u_t$
where $i_t$ is the inclusion  $i_t\colon X_t\hookrightarrow \sX$. Following standard notation, we denote by $\sA_{\sX}^{n,0}(\varphi^*\omega_S^\vee)$ the sheaf of smooth $(n,0)$-forms on $\sX$ with values in $\varphi^*\omega_S^\vee$.

Since $\bu$ is holomorphic on the fibers, we write its derivatives as
\begin{equation}
\label{eqbo}
\bar{\partial}\bu=\nu\wedge d\bar{t}+\eta\wedge dt
\end{equation} and
\begin{equation}\label{fi}
\partial^\phi\bu=e^\phi\partial e^{-\phi}\bu=\mu\wedge dt.
\end{equation} In \cite{bern1} it is shown that
\begin{equation}\label{antiholcon}
(D''_Fu)_t=\nu_t d\bar{t}
\end{equation} and 
\begin{equation}\label{holcon}
(D_F'u)_t=P(\mu_t)dt
\end{equation} where $P$ is the orthogonal projection of $(n, 0)$-forms on the space of holomorphic $(n, 0)$-forms.

It turns out that the $(n-1,1)$-form $\eta_t$ has a neat interpretation in terms of the cup product with the Kodaira-Spencer class $\xi_t\in H^1(X_t,T_{X_t})$ associated to the deformation of $X_t$ given by $f\colon \sX\to \Delta$. In fact $\eta_t$ defines the cohomology class of the cup product $\xi_t\cup u_t\in H^{n-1,1}(X_t,g^*_t\omega_S^\vee)$, see \cite[Lemma 2.2]{bern2}. This result shows the connection between deformations of $X_t$ and metric properties of $F$. In the next section we will explore the analogous relation between deformation of maps $g_t\colon X_t\to S$ and metric properties of $E$.

Of course the choice of the representative $\bu$ is not unique; two such choices differ by a term of the form $v\wedge dt$ where $v$ is an $(n-1,0)$-form with values in $\varphi^*\omega_S^\vee$, i.e. smooth sections of the kernel in Sequence (\ref{approcciobo}). As a consequence, the sections $\nu$, $\eta$, $\mu$ depend on this choice; on the other hand the Formulas (\ref{antiholcon}) and (\ref{holcon}) do not. 

The curvature formula for the hermitian metric is
\begin{equation}
\label{curvatura0}
\langle\Theta_F u_t,u_t \rangle_t=-\lVert P_\perp\mu_t\lVert^2_t+f_*(c_ni\partial\bar{\partial}\phi\wedge\bu\wedge\bar{\bu}e^{-\phi})/dV_t-c_n\int_{X_t}\eta_t\wedge\bar{\eta}_te^{-\phi}.
\end{equation}  Here $P_\perp$ is the orthogonal projection on the orthogonal complement of holomorphic forms, $dV_t:=idt\wedge d\bar{t}$ and $f_*$ is the push-forward of $(n,n)$-forms defined by 
$$
\int_\Delta f_*(\alpha)\wedge\beta=\int_X\alpha\wedge f^*(\beta) 
$$ for $\alpha$ a form on $\sX$ and $\beta$ a form on $\Delta$.

Note that in Formula (\ref{curvatura0}), the left member does not depend on the choice of the representative $\bu$, while the right member does. Indeed the choice of the right representative is essential to make Formula (\ref{curvatura0}) more informative.
In particular if $u$ is a holomorphic section of $F$, by \cite{bern1}, \cite[Lemma 2.1]{bern2}, there is a preferred choice of representative $\bu$ such that $\mu_t$ is holomorphic on $X_t$ and $\eta_t$ is primitive on $X_t$. With this choice Formula (\ref{curvatura0}) becomes
\begin{equation}
\label{curvatura1}
\langle\Theta_F u_t,u_t \rangle_t=f_*(c_ni\partial\bar{\partial}\phi\wedge\bu\wedge\bar{\bu}e^{-\phi})/dV_t+\lVert\eta_t\lVert^2.
\end{equation}

For the details on this Formula and its implications see \cite{bern1,bern2}, here we just point out that it comes from the fact that, for the above mentioned preferred representative $\bu$, $P_\perp\mu_t=0$ since $\mu_t$ is holomorphic and that $-c_n\int_{X_t}\eta_t\wedge\bar{\eta}_t=\lVert\eta_t\lVert^2$ since $\eta_t$ is primitive. 


\subsection{Representatives of sections of $E$}\label{subsez}
We go back to the subbundle $E$ of $F$ with the aim of computing the Chern connection $D_E= D_E'+ D_E''$ and curvature $\Theta_E$. By Proposition \ref{normeuguali}, we know that the Chern connection $D_E$ is induced by the connection $D_F$ on $F$ just described, but, as we have just seen in the case of $F$, the choice of a suitable representative of a section is a key factor to obtain more informative curvature formulas.

In fact, following the two interpretations of Remark \ref{keyrmk}, we  have two natural choices of representatives of a section $u$ of $E$ and depending on this choice we will obtain different connection and curvature formulas for $E$.
\begin{enumerate}
	\item if we regard $u_t$ as a global holomorphic $(n,0)$-form on $X_t$ with values in $g_t^*\omega_S^\vee$, then a suitable representative $\bu$ is a smooth $(n,0)$-form on $\sX$ with values in $\varphi^*\omega_S^\vee$. This corresponds to a smooth splitting of the first row of Diagram \ref{diagrammaconfronto}. 
	\item if we regard $u_t$ as a global relative holomorphic $(n-1)$-form for the morphism $g_t$, then a suitable representative $\tu$  is a smooth relative $(n-1,0)$-form for the fibration $\varphi$. This is obtained thanks to the second row of Diagram \ref{diagrammaconfronto} (after removing torsion).
\end{enumerate} 
 We will denote the sheaf of these smooth relative $(n-1,0)$-form by $\sA_{\sX/S}^{n-1,0}$. As before, the choice of the representative $\tu$ is not unique and two such choices differ by a term of the form $v\wedge dt$ where this time $v$ is a relative $(n-2)$-form, c.f. the kernel of the second row of Diagram \ref{diagrammaconfronto}.

Consider the middle vertical morphism in Diagram (\ref{diagrammaconfronto}) for smooth sections: 
\begin{equation}\label{rho}
\rho\colon \sA_{\sX/S}^{n-1,0}\to \sA_{\sX}^{n,0}(\varphi^*\omega_S^\vee);
\end{equation}  we recall it is given by wedge product.

Note that this morphism is not surjective, in particular not every representative $\bu$ is the image of some $\tu$; from now on we will only consider those that are. 
Hence from now on we denote $\btu:=\rho(\tu)$. 

\begin{rmk}\label{ex}
We give an idea of how things can be written locally on an open subset on $\sX$ where the fibration $G$ is smooth. Take $x_1,\dots,x_{n-1},s,t$ local coordinates so that $G$ is given by projection on the last two coordinates which are coordinates of $S$ and $\Delta$ respectively. Then $\tu$ is written as $$\tu=adx_1\wedge\dots\wedge dx_{n-1}+\sum_ib_i dx_1\wedge\dots\wedge\widehat{dx_i}\wedge\dots\wedge dx_{n-1}\wedge dt,$$ where $\widehat{dx_i}$ means that $dx_i$ does not appear in the wedge product. Its image $\btu=\rho(\tu)$ is $$\btu=adx_1\wedge\dots\wedge dx_{n-1}\wedge ds\otimes \frac{\partial}{\partial s}+\sum_ib_i dx_1\wedge\dots\wedge\widehat{dx_i}\wedge\dots\wedge dx_{n-1}\wedge dt\wedge ds\otimes \frac{\partial}{\partial s}.$$ Note that after restriction on the fibers $X_t$ they correspond exactly to the two different interpretations of $u_t$ as in Remark \ref{remarkloc}. On the other hand note that another representative $\bu$ not in the image of $\rho$ is in principle of the form $$\bu=adx_1\wedge\dots\wedge dx_{n-1}\wedge ds\otimes \frac{\partial}{\partial s}+\sum_ib_i dx_1\wedge\dots\wedge\widehat{dx_i}\wedge\dots\wedge dx_{n-1}\wedge dt\wedge ds\otimes \frac{\partial}{\partial s}+c dx_1\wedge\dots\wedge dx_{n-1}\wedge dt\otimes \frac{\partial}{\partial s}.$$

Hence the key difference between a general representative $\bu$ and a representative $\btu$ in the image of $\rho$ is that locally we have that $\btu\wedge ds=0$ while $\bu\wedge ds\neq0$.
%
%
\end{rmk}

\begin{nota}\label{nota}
From now on we fix the notation above, that is, given a section $u$ of $E$, we will denote by $\tu$ a smooth representative of $u$ in $\sA_{\sX/S}^{n-1,0}$, and  by $\btu$ its image in $\sA^{n,0}(\varphi^*\omega_S^\vee)$ via $\rho$.

In particular $\btu$ can be locally expressed as in Remark \ref{ex} and  $\btu\wedge ds=0$. This is clear since $\rho$ is locally given by wedge with $ds$.
\end{nota}

Since we will use that $\btu\wedge ds=0$ many times in the following, we write it here as a slightly more general easy Lemma for future reference. 

 \begin{lem}\label{wedge}Let $h\colon Y\to Z$ be a fibration over a curve and consider the morphisms
  $$
\rho_{i}\colon \sA^{i,0}_{Y/Z}\to \sA_Y^{i+1,0} (h^*\omega_Z^\vee)
$$  
  given by the wedge product. Call $\alpha$ a relative form in $\sA^{i,0}_{Y/Z}$ and $\beta:=\rho_{i}(\alpha)$.

 We have that $\beta\wedge h^*\gamma=0$ for every smooth form $\gamma$ in $\sA^{1,j}_Z$, $j=0,1$.
 \end{lem}

\subsection{Connection and curvature of $E$ using the representative $\btu$}
 We have the following result on the connection on $E$ using the representative $\btu=\rho(\tu)$.
\begin{prop}\label{primaver}
Take $u$ a section of $E$ and $\btu$ a representative in $\sA^{n,0}(\varphi^*\omega_S^\vee)$ as above.
The (1,0)-part of the Chern connection of $E$ is given by $(D'_Eu)_t=P'({\mu_t})dt$ where $\partial\btu={\mu}\wedge dt$ and $P'$ is the orthogonal projection of smooth $(n, 0)$-forms onto $E_t$. 
The (0,1)-part of the Chern connection is given by $(D''_Eu)_t={\nu_t}d\bar{t}$ where $\bar{\partial}\btu={\nu}\wedge d\bar{t}+{\eta}\wedge dt$. 
\end{prop}
\begin{proof}
These formulas can be derived by the expression of the Chern connection for $F$ recalled in (\ref{antiholcon}) and (\ref{holcon}). 

For $D'_E$, write 
$$
\partial^\phi\btu=e^\phi\partial e^{-\phi}\btu=\mu\wedge dt
$$ as in (\ref{fi}) and note that since $\phi$ is pullback of a local function on $S$, by Lemma \ref{wedge}, $\partial^\phi\btu=\partial\btu$ for our choice of $\btu=\rho(\tu)$. Now $D'_F$ is given by $$D_F'u=P(\mu)dt$$ where $P$ is the orthogonal projection of $(n, 0)$-forms on the space of holomorphic $(n, 0)$-forms, hence $D'_E$ is given by the orthogonal projection of $P(\mu)$ into $E_t$. The composition of these two orthogonal projections is $P'$.

The formula for $D''_E$ is trivial from (\ref{antiholcon}) since $E\subseteq F$ is a holomorphic subbundle.
\end{proof} 
\begin{rmk}
	In particular note that the $(1,0)$-part of the connection is not related to the differential $\partial^\phi$ as in (\ref{holcon}), but to the standard holomorphic differential $\partial$. This is due to the particular choice of the representative $\btu$.
\end{rmk}
By the same reasoning we also have 
\begin{cor} 
\label{secondaff}
With the same notation, the second fundamental form of $E$ in $F$ is given by $(P({\mu})-P'({\mu}))dt$.
\end{cor}

For the curvature, consider a section $u$ of $E$. We fix again a smooth representative $\btu$ of $u$ as above.
With this choice of representatives, we have a simplification of Formula (\ref{curvatura0}):
\begin{equation}\label{curvF}
\langle\Theta_F u_t,u_t \rangle_t=-\lVert P_\perp\mu_t\lVert^2_t-c_n\int_{X_t}\eta_t\wedge\bar{\eta}_te^{-\phi}
\end{equation}
This comes from the fact that $\partial\bar{\partial}\phi\wedge\btu\wedge\overline{\btu}=0$ by Lemma \ref{wedge}.
\begin{rmk}
	Both Formula (\ref{curvatura1}) and Formula (\ref{curvF}) give a simplification of the general Formula (\ref{curvatura0}). In both of this cases, one of the summands disappears, depending on the choice of the representative of $u$. 
\end{rmk}
Using the same notation of Proposition \ref{primaver}, we get the following formula for the curvature of $E$.
\begin{thm}\label{curv1}
The curvature of $E$ on $u$ is given at $t$ by 
\begin{equation*}
\langle\Theta_E u_t,u_t \rangle_t=-\lVert P'_\perp{\mu_t}\lVert^2_t-c_n\int_{X_t}\eta_t\wedge\bar{\eta}_te^{-\phi}
\end{equation*} where $P'_\perp$ is the orthogonal projection of $(n,0)$-forms on the orthogonal complement of $E_t$.
\end{thm}\begin{proof}
By (\ref{curvF}) and the explicit formulation of the second fundamental form in Corollary \ref{secondaff}, we have that 
\begin{equation*}
\langle\Theta_E u_t,u_t \rangle_t=-\lVert P_\perp\mu_t\lVert^2_t-c_n\int_{X_t}\eta_t\wedge\bar{\eta}_te^{-\phi}-\lVert (P-P')({\mu_t})\lVert^2
\end{equation*} Now we conclude since $\lVert P_\perp\mu_t\lVert^2+\lVert (P-P')({\mu_t})\lVert^2=\lVert P'_\perp{\mu_t}\lVert^2$.
\end{proof}
\begin{rmk}
In \cite{bern1,bern2} the author shows that the curvature of the vector bundle associated to the direct image $f_*(\omega_{\sX/\Delta}\otimes \sL)$ is semipositive when $\sL$ is a line bundle  equipped with a smooth metric of
semipositive curvature. In our case $F$ is the vector bundle associated to $f_*\omega_{\sX/\Delta\times S}=f_*(\omega_{\sX/\Delta}\otimes \varphi^*\omega_S^\vee)$. If $S$ is a curve of general type then $\sL$ is seminegative, so in general we cannot expect the curvature of $F$ and of  $E$ to be semipositive.
\end{rmk}
\subsection{Connection and curvature of $E$ using the representative $\tu$ and relative differentials} Proposition \ref{primaver} and Theorem \ref{curv1} give formulas for the connection and curvature of $E$ using the representative $\btu$ and its derivatives. It is a natural question to ask if we can interpret the above formulas in terms of the representative $\tu$, which we recall is a smooth relative form in $\sA_{\sX/S}^{n-1,0}$, and its relative differentials.

Recall that, given a holomorphic fibration of smooth complex manifolds $h\colon Y\to Z$, the relative differential is usually denoted by $d_{Y/Z}$:
$$
d_{Y/Z}\colon \sA^p_{Y/Z}\to \sA^{p+1}_{Y/Z}.
$$ By $d^0_{Y/X}$ we denote its composition with the quotient by the torsion part 
$$
d_{Y/Z}^0\colon \sA^p_{Y/Z}\to \sA^{p+1}_{Y/Z}\to \sA^{p+1}_{Y/Z}/\text{Tors}.
$$
In the same fashion we will consider $\partial_{Y/Z}$, $\bar{\partial}_{Y/Z}$, $\partial_{Y/Z}^0$ and $\bar{\partial}^0_{Y/Z}$.
 
Consider $u$ a section of $E$ and $\tu$, $\btu$ representatives as in Notation \ref{nota}; the first natural question is to compare $\partial_{\sX/S}\tu$ and $\bar{\partial}_{\sX/S}\tu$ with $\partial \btu$ and $\bar{\partial}\btu$ respectively.

For the anti-holomorphic derivative, we note that $\bar{\partial}_{\sX/S}\tu$ is a smooth section in $\sA^{n-1,1}_{\sX/S}$ and $\bar{\partial}\btu$ in $\sA^{n,1}_{\sX}(\varphi^*\omega^\vee_S)$.
We can write as in Proposition \ref{primaver}
\begin{equation}\label{etauno}
\bar{\partial}\btu=\nu\wedge d\bar{t}+\eta\wedge dt
\end{equation}  and also
\begin{equation}\label{etatilde}
\bar{\partial}^0_{\sX/S}\tu=\tilde{\nu}\wedge d\bar{t}+\tilde{\eta}\wedge dt.
\end{equation}
Note that $\nu$  and $\tilde{\nu}$ define smooth sections of $\sA^{n,0}_{\sX/\Delta}( \varphi^*\omega^\vee_S)$ and $\sA^{n-1,0}_{\sX/\Delta\times S}/\textnormal{Tors}$ respectively and $\eta$  and $\tilde{\eta}$ define smooth sections of $\sA^{n-1,1}_{\sX/\Delta}( \varphi^*\omega^\vee_S)$ and $\sA^{n-2,1}_{\sX/\Delta\times S}/\textnormal{Tors}$.
 We will analyse $\eta$  and $\tilde{\eta}$ in the next section, here we observe that a local computation shows that $\nu$ is the image of $\tilde{\nu}$ via the inclusion
 $$
 \sA^{n-1,0}_{\sX/\Delta\times S}/\textnormal{Tors}\hookrightarrow \sA^{n,0}_{\sX/\Delta}( \varphi^*\omega^\vee_S)
$$ given by the wedge product.

The restrictions of $\tilde{\nu}$ and $\nu$ to the fibers $X_t$ are independent of any choice of representative so the $(0,1)$-part of the Chern connection of $E$ which  is expressed in Proposition \ref{primaver} thanks to the usual derivative and $\nu$, can also be written in terms of the relative derivative and $\tilde{\nu}$ as follows:
\begin{prop}
\label{eta}
The $(0,1)$-part of the Chern connection of $E$ is given by the formula $(D''_Eu)_t=\tilde{\nu}_td\bar{t}$ where $\tilde{\nu}$ is defined by $\bar{\partial}^0_{\sX/S}\tu=\tilde{\nu}\wedge d\bar{t}+\tilde{\eta}\wedge dt$. 
\end{prop}
For the holomorphic derivative, we have that ${\partial}_{\sX/S}\tu$ is a smooth section in $\sA^{n,0}_{\sX/S}$ and ${\partial}\btu$ in $\sA^{n+1,0}_{\sX}(\varphi^*\omega^\vee_S)$. They are related via the morphism
$$
\rho_{n}\colon \sA^{n,0}_{\sX/S}\to\sA^{n+1,0}_{\sX}(\varphi^*\omega^\vee_S)
$$ that is $\rho_{n}(\partial_{\sX/S}\tu)=\partial \btu$. This can be checked, for example, by local computation. 
%
However, while we can always write 
$$
\partial \btu=\mu\wedge dt
$$ as in Proposition \ref{primaver}, in general it is not true that we can express $\partial^0_{\sX/S}\tu$ as
\begin{equation*}
\partial^0_{\sX/S}\tu=\tilde{\mu}\wedge dt
\end{equation*} where $\tilde{\mu}$ is a smooth relative $(n-1)$-form for the fibration $G\colon \sX\to \Delta\times S$.  

This is true however in some cases, for example if $G$ is smooth or $G$ is locally as in Case (3) of Example \ref{example1}. We refer to these cases simply by saying that $\varphi$ is semistable with separated variables, since we can choose coordinates such that $t$ does not appear in the local expression of $\varphi$. In this case in fact we have the following lemma.
\begin{lem}
\label{isotors}
If $\varphi$ is semistable with separated variables, we have an isomorphism $\Omega^{n}_{\sX/S}/\textnormal{Tors}\cong (\Omega^{n-1}_{\sX/\Delta\times S}/\textnormal{Tors})\otimes f^*\omega_{\Delta}$.
\end{lem}
\begin{proof}
Consider the morphism given by wedge product
$$
(\Omega^{n-1}_{\sX/\Delta\times S}/\textnormal{Tors})\otimes f^*\omega_{\Delta}\to \Omega^{n}_{\sX/S}/\textnormal{Tors}
$$ Note that these sheaves are generically of rank one and torsion free, hence this morphism is injective. We just need to show the surjectivity.

We show this locally around the critical points, since it is trivial where $G$ is smooth. Take local coordinates $x_1,\dots,x_{n},t$ on $\sX$ such that $\varphi$ is given locally by  $x_1,\dots,x_{n},t\mapsto x_1\cdots x_p$ since we are in Case (3) of Example \ref{example1}.

The sheaf $\Omega^{n}_{\sX/S}/\textnormal{Tors}$ is locally generated by the wedge products
$$
dx_1\wedge\dots\wedge\widehat{dx_i}\wedge\dots\wedge dx_p\wedge dx_{p+1}\wedge\dots\wedge dx_n\wedge dt,
$$ $i=1,\dots,p$, which are related by the relation coming from $d(x_1\cdots x_p)=0$. It is easy to see that these are in the image of  $(\Omega^{n-1}_{\sX/\Delta\times S}/\textnormal{Tors})\otimes f^*\omega_{\Delta}$ and this shows the required surjectivity.

As a closing note, the torsion of $\Omega^{n}_{\sX/S}$ is given by the element
$$
dx_1\wedge\dots\wedge dx_n
$$ which is not zero but vanishes when multiplied (for example) by $x_1\cdots x_{p-1}$. This element is not in the image of 
$$
(\Omega^{n-1}_{\sX/\Delta\times S}/\textnormal{Tors})\otimes f^*\omega_{\Delta}\to \Omega^{n}_{\sX/S}
$$ and not even of 
$$
\Omega^{n-1}_{\sX/\Delta\times S}\otimes f^*\omega_{\Delta}\to \Omega^{n}_{\sX/S}
$$ and this shows that taking the quotient by the torsion part is necessary.
\end{proof}

Hence in the semistable case we have the following formula for the $(1,0)$-part of the Chern connection.
\begin{prop}
\label{propanti}
If $\varphi$ is semistable with separated variables, the $(1,0)$-part of the Chern connection is given by the formula $(D'_Eu)_t=P'(\tilde{\mu}_t)dt$ where $\partial^0_{\sX/S}\tu=\tilde{\mu}\wedge dt$ and $P'$ is the orthogonal projection  onto $E_t$. 
\end{prop}
\begin{proof}
Thanks to the isomorphism of Lemma \ref{isotors} in the case of smooth sections, we can write $\partial^0_{\sX/S}\tu$ as 
\begin{equation*}
\partial^0_{\sX/S}\tu=\tilde{\mu}\wedge dt
\end{equation*} where $\tilde{\mu}$ is a smooth section in $\sA^{n-1,0}_{X/\Delta\times S}/\textnormal{Tors}$.

On the other hand, writing 
$$
\partial\btu=\mu\wedge dt
$$ as in Proposition \ref{primaver}, it is not difficult to see that  $\mu$ is the image of $\tilde{\mu}$  via the morphism 
$$
 \sA^{n-1,0}_{\sX/\Delta\times S}/\textnormal{Tors}\hookrightarrow \sA^{n,0}_{\sX/\Delta}( g^*\omega^\vee_S).
$$


Comparing with Proposition \ref{primaver} we have the result.
\end{proof}

We finally have the following curvature formula on $E$ in terms of $\tu$ and its relative differentials.  This is Theorem \ref{thm1} of the Introduction.

\begin{thm}\label{curv2}
 Take $u$ a section of $E$ and write as before $\partial\btu=\mu\wedge dt$ and $\bar{\partial}^0_{\sX/S}\tu=\tilde{\nu}\wedge d\bar{t}+\tilde{\eta}\wedge dt$. Then the curvature of $E$ on $u$ is given at $t$ by 
 \begin{equation}\label{curvrel1}
 	\langle\Theta_E u_t,u_t \rangle_t=-\lVert P'_\perp{\mu_t}\lVert^2_t-2c_{n-1}\int_{X_t}\tilde{\eta}_t\wedge\overline{\tilde{\eta}_t}\wedge g^*_t\textnormal{Vol}S
 \end{equation}  where $P'_\perp$ is the orthogonal projection on the orthogonal complement of $E_t$.
   If $\varphi$ is semistable with separated variables, we can also write  $\partial^0_{\sX/S}\tu=\tilde{\mu}\wedge dt$ and the curvature becomes
\begin{equation}
\langle\Theta_E u_t,u_t \rangle_t=-\lVert P'_\perp\tilde{\mu}_t\lVert^2_t-2c_{n-1}\int_{X_t}\tilde{\eta}_t\wedge\overline{\tilde{\eta}_t}\wedge g^*_t\textnormal{Vol}S.
\end{equation}
\end{thm}
\begin{proof}
We use the formula of Theorem \ref{curv1}:
\begin{equation*}
	\langle\Theta_E u_t,u_t \rangle_t=-\lVert P'_\perp{\mu_t}\lVert^2_t-c_n\int_{X_t}\eta_t\wedge\bar{\eta}_te^{-\phi}
\end{equation*}

 We have already discussed in Proposition \ref{propanti} the relation between $\mu$ and $\tilde{\mu}$ when $\varphi$ is semistable, so we just need to look at the second summand.
 
 For this, it is not difficult to see, with local computations similar to those of Proposition \ref{normeuguali},  that on $X_t$ 
 $$
 c_n\eta_t\wedge\bar{\eta}_te^{-\phi}=2c_{n-1}\tilde{\eta}_t\wedge\overline{\tilde{\eta}_t}\wedge g^*_t\textnormal{Vol}S.
 $$
 \end{proof}

\section{Curvature and deformations}
 \label{sez5}
 
In this section we relate the curvature formulas from Section \ref{sec4} to the theory of deformation of maps with fixed target. We also identify a subbundle of $E$ with seminegative curvature. 

We assume for simplicity that $n=2$, that is $\sX$ has dimension 3 and each $g_t\colon X_t\to S$ is a fibred surface over a fixed curve.
Most of the results generalise to the higher dimensional general case, however the discussion becomes technically more complicated without giving substantial improvement.

 As a first technical simplification for example, since the sheaves of relative 1-forms are torsion free for semistable fibrations, the vector bundle $E$ is just associated  to $f_*(\Omega^{1}_{\sX/\Delta\times S})$ and its fiber is $E_t=\Ima(H^0(X_t,\Omega^{1}_{X_t/S})\to H^0(X_t,\omega_{X_t/S}))$; comparing with Definition (\ref{notationE}) and Formula (\ref{Efiber}) we see that we do not need to remove torsion.

As discussed in Section \ref{sec2}, the deformation of maps with fixed target given by $f\colon \sX\to \Delta\times S\to \Delta$ defines elements $\zeta_t\in \Ext^1(\Omega^1_{X_t/S},\sO_{X_t})$, for $t\in \Delta$. We denote as before by $\xi_t$ the image of $\zeta_t$ in $H^1(X_t,T_{X_t})=\Ext^1(\Omega^1_{X_t},\sO_{X_t})$ via the homomorphism (\ref{defmap}), so that $\xi_t$ is the usual Kodaira-Spencer class of the deformation of $X_t$ induced by $f$.
Since $\zeta_t$ and $\xi_t$ are elements of extension groups, their interpretation as short exact sequences is as follows.

First, $\zeta_t\in \Ext^1(\Omega^1_{X_t/S},\sO_{X_t})$ corresponds to the exact sequence
\begin{equation}\label{zeta}
	0\to \sO_{X_t}\overset{dt}{\underset{}{\to}} \Omega^1_{\sX/S|X_t}\to \Omega^1_{X_t/S}\to 0
\end{equation} which is the restriction to $X_t$ of (\ref{1forme}). On the other hand,  $\xi_t\in \Ext^1(\Omega^1_{X_t},\sO_{X_t})$ corresponds to 
\begin{equation}\label{xi}
	0\to \sO_{X_t}\overset{dt}{\underset{}{\to}} \Omega^1_{\sX|X_t}\to \Omega^1_{X_t}\to 0.
\end{equation} Actually, since $\Ext^1(\Omega^1_{X_t},\sO_{X_t})\simeq \Ext^1(\omega_{X_t}\otimes g_t^*\omega_S^\vee,\Omega^1_{X_t}\otimes g_t^*\omega_S^\vee)$, $\xi_t$ also corresponds to the exact sequence obtained by the second wedge product of (\ref{xi}) followed by tensor product with the line bundle $g_t^*\omega_S^\vee$:
\begin{equation}\label{xi2}
	0\to \Omega^1_{X_t}\otimes g_t^*\omega_S^\vee \to \Omega^2_{\sX|X_t}\otimes g_t^*\omega_S^\vee\to \omega_{X_t}\otimes g_t^*\omega_S^\vee\to 0.
\end{equation}
Sequences (\ref{zeta}) and (\ref{xi2}) fit into the commutative diagram
\begin{equation}\label{diag}
	\xymatrix{
		0\ar[r]& \Omega^1_{X_t}\otimes g_t^*\omega_S^\vee \ar[r]&  \Omega^2_{\sX|X_t}\otimes g_t^*\omega_S^\vee\ar[r]&  \omega_{X_t}\otimes g_t^*\omega_S^\vee\ar[r]&  0\\
		0\ar[r]& \sO_{X_t}\ar[r]\ar[u]&  \Omega^1_{\sX/S|X_t}\ar[r]\ar[u]& \Omega^1_{X_t/S}\ar[r]\ar[u]&0
	}
\end{equation}
 which is the restriction of Diagram (\ref{diagrammaconfronto}) to the fiber $X_t$ in our case (i.e. $n=2$).
 We denote by $\xi_t\cup$ the cup product with the class $\xi_t$, which gives the coboundary map for the cohomology long exact sequence of the top row of Diagram (\ref{diag}). In the same fashion $\zeta_t\cup$ will denote the coboundary map for the cohomology sequence of the bottom row.

In this setting we analyse in details the relation between the relative forms $\eta$ and $\tilde{\eta}$ defined in (\ref{etauno}) and (\ref{etatilde}) respectively.

By \cite{bern1,bern2}, we know that $\eta_t$ defines the cohomology class of the cup product  $\xi_t\cup u_t\in H^{1,1}(X_t,g^*_t\omega_S^\vee)$, hence we show the connection between $\tilde{\eta}_t$ and the class $\zeta_t$ parametrising the deformation of maps with fixed target. The first step is the following result which is immediate from the commutativity of Diagram (\ref{diag})
 \begin{prop}\label{firstprop}
Take a section $u$ of $E$. The image of the cohomology class $\zeta_t\cup u_t\in H^1(X_t,\sO_{X_t})$ via the homomorphism $H^1(X_t,\sO_{X_t})\to  H^1(X_t,\Omega^1_{X_t}\otimes g_t^*\omega_S^\vee)=H^{1,1}(X_t,g^*_t\omega_S^\vee)$ is the class $\xi_t\cup u_t=[\eta_t]$.
\end{prop}

So in particular, the cohomology class $[\eta_t]$ is in the image of $H^1(X_t,\sO_{X_t})\to  H^{1,1}(X_t,g^*_t\omega_S^\vee)$.
When $\eta_t$ is obtained from the representative $\btu$ as in (\ref{etauno}), we can read this also at the level of sections and not only of cohomology classes. In fact $\eta_t$ is a smooth section in $\sA^{1,1}_{X_t}(g_t^*\omega_S^\vee)$ and there is an injective morphism with torsion free quotient
$$
\sA^{0,1}_{X_t}\hookrightarrow\sA^{1,1}_{X_t}(g_t^*\omega_S^\vee)
$$ given by wedge product. It is not difficult to see that $\eta_t$ locally on the good open set where $g_t$ is smooth is in the image of this morphism, hence it is globally in the image since the quotient has no torsion.  We stress that this relies on the fact that the representative $\btu$ has the particular form shown in Notation \ref{nota}. We denote by $\hat{\eta}_t$ the section of $\sA^{0,1}_{X_t}$ which maps to $\eta_t$. Note that by Proposition \ref{firstprop}, $\zeta_t\cup u_t=[\hat{\eta}_t]$.

On the other hand, recalling the definition of $\tilde{\eta}$ from (\ref{etatilde}), $\tilde{\eta}_t$ is a smooth section in $\sA^{0,1}_{X_t/S}$. There is no direct morphism between $\sA^{1,1}_{X_t}(g_t^*\omega_S^\vee)$ and $\sA^{0,1}_{X_t/S}$, nevertheless they are related thanks to $\sA^{0,1}_{X_t}$ as follows
\begin{equation}\label{diagsh}
	\xymatrix{
	\sA^{0,1}_{X_t}\ar@{^{(}->}[r]\ar[rd]&\sA^{1,1}_{X_t}(g_t^*\omega_S^\vee)\\
	&\sA^{0,1}_{X_t/S}
	}
\end{equation} where the diagonal arrow is the quotient with kernel the smooth (0,1)-forms coming from $S$.

By local computation, it is not difficult to see that the image of $\hat{\eta}_t$ in $\sA^{0,1}_{X_t/S}$ is exactly  $\tilde{\eta}_t$; that is
\begin{equation*}
	\xymatrix{
	\hat{\eta}_t	\ar@{|->}[r]\ar@{|->}[rd]&\eta_t\\
		&\tilde{\eta}_t
	}
\end{equation*} 

Now $u_t\in E_t$ can be seen as a relative 1-form for the morphism $g_t$ (this is the second interpretation from Remark \ref{keyrmk}); hence $u_t$ associates to a general $s\in S$ a holomorphic 1-form on the fiber $X_{t,s}:=g_t^{-1}(s)$, we denote this form by  $u_{t,s}\in H^0(X_{t,s}, \Omega^1_{X_{t,s}})$. 
Similarly, $\tilde{\eta}_t$ is a relative $(0,1)$-form for the morphism $g_t$, so we denote by $\tilde{\eta}_{t,s}$ the $(0,1)$-form on the fiber $X_{t,s}$ given by $\tilde{\eta}_t$.

The above discussion gives the following result which finally relates $\tilde{\eta}_t$ with the deformation of maps with fixed target.
\begin{prop}
The form $\tilde{\eta}_t$ defines a global holomorphic section of  $R^1{g_t}_*\sO_{X_t}$. This section associates to the general $s\in S$ the class $\phi_s(\zeta_t)\cup u_{t,s}$, where $\phi_s$ is as in (\ref{phi}). In particular $[\tilde{\eta}_{t,s}]=\phi_s(\zeta_t)\cup u_{t,s}$ as elements in $H^1(X_{t,s},\sO_{X_{t,s}})$.
\end{prop}
\begin{proof}
The Leray spectral sequence gives the short exact sequence 
\begin{equation}\label{leray}
	0\to H^1(S,\sO_S)\to H^1(X_t,\sO_{X_t})\to H^0(S,R^1{g_t}_*\sO_{X_t})\to 0.
\end{equation}
Since the forms $\eta_t$, $\hat{\eta}_t$ define cohomology classes, we consider the analogue of Diagram (\ref{diagsh})  in cohomology
	\begin{equation*}
		\xymatrix{
			H^1(X_t,\sO_{X_t})\ar[r]\ar[rd]&H^{1,1}(X_t,g^*_t\omega_S^\vee)\\
			&H^0(S,R^1{g_t}_*\sO_{X_t})
		}
	\end{equation*} Here the horizontal arrow is the homomorphism considered in Proposition \ref{firstprop} and the diagonal arrow is given in (\ref{leray}). 
	
	We have seen that the horizontal arrow maps 
	$$
	\zeta_t\cup u_t=[\hat{\eta}_t]\mapsto \xi_t\cup u_t=[\eta_t]
	$$  while the image of $\zeta_t\cup u_t=[\hat{\eta}_t]$ via the diagonal arrow is a section $[\tilde{\eta}_t]\in H^0(S,R^1{g_t}_*\sO_{X_t})$ defined by $\tilde{\eta}_t$ by the previous discussion.
	
	From Lemma \ref{ss2}, $\zeta_t$ gives an element of $H^0(S,\ext^1_{g_t}(\Omega^1_{X_t/S},\sO_{X_t}))$ and $R^1{g_t}_*\sO_{X_t}$  is a vector bundle on $S$ (see \cite[Theorem 2.6]{K}) with general fiber $H^1(X_{t,s},\sO_{X_{t,s}})$, hence $[\tilde{\eta}_t]$ can be interpreted as a section of this vector bundle which associates to the general $s\in S$ the class $\phi_s(\zeta_t)\cup u_{t,s}\in H^1(X_{t,s},\sO_{X_{t,s}})$.
\end{proof}

Since $\tilde{\eta}_{t,s}$ is a $(0,1)$-form on the fiber $X_{t,s}$ in particular it is primitive. In this case $-c_{n-1}\tilde{\eta}_{t,s}\wedge \overline{ \tilde{\eta}_{t,s}}=|\tilde{\eta}_{t,s}|^2$, hence in the curvature formula of Theorem \ref{curv2} the second summand is manifestly non-negative.

\begin{prop}
	\label{curvdef}
	With the same notation from Theorem \ref{curv2}, if $n=2$, the curvature of $E$ on $u$ is given at $t$ by 
	\begin{equation*}
		\langle\Theta_E u_t,u_t \rangle_t=-\lVert P'_\perp{\mu_t}\lVert^2_t{+2c_{n-1}\int_{X_t}|\tilde{\eta}_{t}|^2 g^*_t\textnormal{Vol}S}
	\end{equation*}  where $\tilde{\eta}_t$ is such that $[\tilde{\eta}_{t,s}]=\phi_s(\zeta_t)\cup u_{t,s}$.
\end{prop} This is Proposition \ref{prop2} of the Introduction.
 
Note that by changing the representative $\tu$ of $u$, the corresponding $\tilde{\eta}_t$ ranges through its cohomology class $[\tilde{\eta}_t]\in H^0(S,R^1{g_t}_*\sO_{X_t})$, as in \cite{bern2}.
So if the class $[\tilde{\eta}_t]$ is zero in $H^0(S,R^1{g_t}_*\sO_{X_t})$ one can change the representative $\tu$ and assume that $\tilde{\eta}_t=0$ not only as a class but as a relative form, so that in particular the $\tilde{\eta}_{t,s}$ are zero. In this case the curvature of $E$ is seminegative at $t$.

 The following examples show cases where the class $\zeta_0$ does not vanish, but the curvature of $E$ is
 seminegative at the point $0\in \Delta$.

 \begin{expl}
 	Consider a hyperelliptic curve $C$ of genus at least 3. Thanks to a famous result of M. Noether, see \cite{C} or \cite[Page 117]{ACGH}, it is well known that the infinitesimal Torelli fails for $C$, that is there exists a smooth family $\sC\to \Delta$ with central fiber $C_0=C$ such that its Kodaira-Spencer class at 0 does not vanish, but has vanishing cup product with every differential form in $H^0(C,\omega_C)$.
 	
 	Consider a smooth curve $S$ and call $\sX:=\sC\times S$. The fibrations
 	$$
 	f\colon \sX\to \Delta\times S\to \Delta
 	$$ allow us to interpret the projection $C_t\times S\to S$ as a deformation of the projection $C_0\times S\to S$ leaving the target fixed.
 	Note that in this example $\zeta_0\in \Ext^1(\Omega^1_{C_0\times S/S},\sO_{C_0\times S})$ is not zero since for example $\phi_s(\zeta_0)\in \Ext^1(\Omega^1_{C_0},\sO_{C_0})$ is exactly the non vanishing Kodaira-Spencer class of $\sC\to \Delta$ at 0 for every $s$. The cup product with $\phi_s(\zeta_0)$ on the other hand is always zero, so by Proposition \ref{curvdef} and the above remarks, the curvature of $E$ at $0\in \Delta$ is seminegative.
 \end{expl}
 \begin{expl}
 Assume that we are in the case of Lemma \ref{ss2} and we can identify $\textnormal{Ext}^{1}(\Omega^1_{X_0/S},\sO_{X_0})\simeq H^0(S,\ext^1_{g_0}(\Omega^1_{X_0/S},\sO_{X_0})).$ Since $\ext^1_{g_0}(\Omega^1_{X_0/S},\sO_{X_0})$ is a sheaf on a curve, it can be decomposed as the direct sum of a torsion sheaf and a torsion free sheaf: $\ext^1_{g_0}(\Omega^1_{X_0/S},\sO_{X_0})=\sT\oplus \sT\sF$. A nontrivial deformation $\zeta_0$ coming from the torsion part $H^0(S,\sT)$ is supported on a Zariski closed set of $S$ hence $\phi_s(\zeta_0)$ is zero for general $s\in S$. In particular by Proposition \ref{curvdef}, the curvature of $E$ at $0\in \Delta$ is seminegative.
 \end{expl}

 \subsection{A seminegativity result}
  
 Proposition \ref{curvdef} gives a seminegativity result for the curvature of an appropriate subbundle of $E$.
 Consider the exact Sequence (\ref{1forme})
\begin{equation*}
	0\to f^*\omega_\Delta\to \Omega^1_{\sX/S}\to \Omega^1_{\sX/\Delta\times S}\to 0
\end{equation*} whose restriction at $X_t$ is the second row of Diagram \ref{diag}. Applying the direct image functor to this sequence we obtain a morphism
$$
f_*(\Omega^1_{\sX/\Delta\times S})\to R^1f_*\sO_{\sX}\otimes \omega_{\Delta}.
$$ Now, recalling that $
f\colon \sX\to  \Delta
$ is the composition of 
$G\colon \sX\to \Delta\times S$ and the projection $p\colon \Delta\times S\to \Delta$, the Grothendieck spectral sequence \cite[Theorem 12.10]{M} gives a morphism 
 $$
  R^1f_*\sO_{\sX}\otimes \omega_{\Delta}\to p_*(R^1G_*\sO_\sX)\otimes \omega_{\Delta}.
  $$
  The composition 
  \begin{equation}\label{comp}
  f_*(\Omega^1_{\sX/\Delta\times S})\to R^1f_*\sO_{\sX}\otimes \omega_{\Delta}\to p_*(R^1G_*\sO_\sX)\otimes \omega_{\Delta}
  \end{equation} is, valued at $t\in \Delta$ and after identification $T_{\Delta,t}\cong \mC$, 
  $$
 H^0(X_t,\Omega^1_{X_t/S})\to H^1(X_t,\sO_{X_t})\to H^0(S,R^1{g_t}_*\sO_{X_t}).
  $$
  This composition gives, on an element $u_t\in H^0(X_t,\Omega^1_{X_t/S})$,
  $$
  u_t\mapsto \zeta_t\cup u_t=[\hat{\eta}_t]\mapsto [\tilde{\eta}_t].
  $$
  \begin{defn}
  	We denote by $K$ the vector bundle associated to the kernel of $
  	f_*(\Omega^1_{\sX/\Delta\times S})\to  p_*(R^1G_*\sO_\sX)\otimes \omega_{\Delta}.
  	$ 
  \end{defn}
  So if we see $u_t\in E_t$ as a relative 1-form for the morphism $g_t$ associating to a general $s\in S$ the form  $u_{t,s}\in H^0(X_{t,s}, \Omega^1_{X_{t,s}})$, the element of $K$ are the sections $u_t$ such that $\phi_s(\zeta_t)\cup u_{t,s}=0\in H^1(X_{t,s},\sO_{X_{t,s}})$. 
  
  As a corollary of Proposition \ref{curvdef} we have Theorem \ref{thm3} of the Introduction.
  \begin{thm}\label{semineg}
  The vector bundle	$K$ is seminegatively curved.
  \end{thm}

    \begin{rmk}
  Consider the map $g_0\colon X_0\to S$. It is well known that the $g_0^*$ gives a homomorphism of Hodge structures $H^1(S,\mC)\to H^1(X_0,\mC)$, hence the quotient $H^0(S,R^1{g_0}_*\mC)$ inherits a Hodge decomposition. Working on $\Delta$ one can consider the associated VHS whose kernel is seminegatively curved by \cite{Z}. In general the bundle $K$ contains this kernel, hence in principle it is not necessarily seminegative.
  \end{rmk}
  \begin{rmk}
  	As pointed out at the beginning of this section, everything is stated in the case of $n=2$ for simplicity, but most of the results still apply in the general case. For example the composition (\ref{comp}) is replaced in the general case by 
  	\begin{equation*}
  		f_*(\Omega^{n-1}_{\sX/\Delta\times S}/{\text{Tors}})\to R^1f_*(\Omega^{n-2}_{\sX/\Delta\otimes S}/{\text{Tors}})\otimes \omega_{\Delta}\to p_*(R^1G_*(\Omega^{n-2}_{\sX/\Delta\otimes S}/{\text{Tors}}))\otimes \omega_{\Delta}
  	\end{equation*}
  	and one can show the seminegativity of its kernel in a similar way.
  \end{rmk}

\end{document}